\documentclass[14pt]{amsart}

\usepackage{amsmath,arydshln,multirow}

\usepackage{arydshln}
\usepackage{cases}
\usepackage{amsmath}
\usepackage{amsfonts}
\usepackage{bm}
\usepackage{arydshln}
\usepackage{amsfonts,amsmath,amssymb,amscd,bbm,amsthm,mathrsfs,dsfont}
\usepackage{mathrsfs}
\usepackage{pb-diagram}
\usepackage{amssymb}
\usepackage{xypic}

\newtheorem{Theorem}{Theorem}[section]
\newtheorem{Lemma}[Theorem]{Lemma}

\newtheorem{Definition}[Theorem]{Definition}
\newtheorem{Corollary}[Theorem]{Corollary}
\newtheorem{Proposition}[Theorem]{Proposition}
\newtheorem{Example}[Theorem]{Example}
\newtheorem{Remark}[Theorem]{Remark}

\newtheorem{Problem}[Theorem]{Problem}

\date{version of \today}

 \normalbaselines
\title{Uniformly column sign-coherence and the existence \\of  maximal green sequences}

\author{Peigen Cao $\;\;\;\;\;\;$ Fang Li $\;\;\;\;\;\;$}
\address{Peigen Cao
\newline Department
of Mathematics, Zhejiang University (Yuquan Campus), Hangzhou, Zhejiang
310027,  P.R.China}
\email{peigencao@126.com}

\address{Fang Li
\newline Department
of Mathematics, Zhejiang University (Yuquan Campus), Hangzhou, Zhejiang
310027, P.R.China}
\email{fangli@zju.edu.cn}

\begin{document}
\renewcommand{\thefootnote}{\alph{footnote}}

\setcounter{footnote}{-1} 
\renewcommand{\thefootnote}{\alph{footnote}}
\setcounter{footnote}{-1} 

\begin{abstract}
 In this paper, we prove that each matrix in $M_{m\times n}(\mathbb Z_{\geq0})$ is uniformly column sign-coherent (Definition \ref{def1} (ii)) with respect to any $n\times n$ skew-symmetrizable integer matrix (Corollary \ref{cor1} (ii)). Using such matrices, we introduce the definition of irreducible skew-symmatrizable matrix (Definition \ref{defind}).  Based on this, the existence of a maximal green
sequence for a skew-symmetrizable matrices is  reduced to the existence of a maximal green sequence for irreducible skew-symmetrizable matrices.
\end{abstract}

\maketitle
\bigskip

\section{introduction}
{\em C-matrices} (respectively, {\em G-matrices}) \cite{FZ3} are important research objects in the theory of cluster algebras. It is known that {\em C-matrices} (respectively, {\em G-matrices}) are column (respectively, row) sign-coherent (see Definition \ref{def1} (i)). In this paper, we consider the matrices which have the similar property  with {\em C-matrices}. This property is called uniformly column sign-coherence (see Definition \ref{def1} (ii)). By the definition of uniformly column sign-coherence and a result in \cite{GHKK} (see Theorem \ref{fterm1} below), we know that $I_n$ is uniformly column sign-coherent using  the terminology in this paper.

 The motivation to consider the uniformly column sign-coherence comes from Proposition \ref{pro1}, which indicates if some submatrix of a skew-symmetrizable $B$ is uniformly column sign-coherent, then there is another submatrix of $B$ is invariant under any particular sequence of mutations.

 It is natural to ask that when a matrix is uniformly column sign-coherent. This is actually a hard question. However, we can turn our mind to the other side to think about how to produce new uniformly column sign-coherent matrices from a given one. Theorem \ref{thm1} in this paper is an answer to this. As a corollary, matrices in  $M_{m\times n}(\mathbb Z_{\geq0})$ are proved to be uniformly column sign-coherence (Corollary \ref{cor1}).

 Maximal green sequences are particular sequences of mutations of
skew-symmetrizable matrices  introduced by Keller \cite{K}. Such particular sequences have numerous applications, including the
computations of spectrums of BPS states, Donaldson-Thomas invariants, tilting of hearts in derived categories,
and quantum dilogarithm identities.

A very important problem in cluster algebra theory is to determine when a given skew-symmetrizable matrix $B$ has a maximal green sequence. In \cite{M} (Theorem 9), Greg Muller proved that if $B$ has a maximal green sequence, so is any principal submatix of $B$. Conversely, if some principal submatrices of $B$ have a maximal green sequence, how about the existence of maximal green sequence of $B$? An answer to this question is given in this paper, based on the discussion of uniformly column sign-coherence. One can refer to Theorem \ref{thm2} for this.

 Thanks to Theorem \ref{thm2} in this paper, and Theorem 9 of \cite{M}, we reduce the existence of a maximal green
sequence for skew-symmetrizable matrices to the existence of a maximal green sequence for irreducible skew-symmetrizable matrices (Definition \ref{defind}). And we give a characteristic for irreducible skew-symmetrizable matrices (Proposition \ref{procycle}).

Note that a special case of Theorem \ref{thm2} has been given in \cite{GM}. The authors proved that if both quivers $Q_1$ and $Q_2$ have a maximal green sequences, then so is the quiver $Q$ which is a $t$-colored direct sum of quivers $Q_1$ and $Q_2$ (Theorem 3.12 of \cite{GM}). And the authors believe this result also holds for any direct sum of $Q_1$ and $Q_2$ (Remark 3.13 of \cite{GM}). Theorem  \ref{thm2} in this paper actually gives an affirm answer to this.

This paper is organized as follows: In Section 2, some basic definitions are given. In Section 3, we give a method to produce uniformly column sign-coherent matrices from a given one (Theorem \ref{thm1}). Thus we prove that each matrix in $M_{m\times n}(\mathbb Z_{\geq0})$ is uniformly column sign-coherent (Corollary \ref{cor1}).
In Section 4, we give the definition of irreducible skew-symmetrizable matrix and their characterization. Then we reduce the existence of a maximal green
sequence for skew-symmetrizable matrices to the existence of a maximal green sequence for irreducible skew-symmetrizable matrices.

\section{Preliminaries}

Recall that an integer matrix $B_{n\times n}=(b_{ij})$ is  called  {\bf skew-symmetrizable} if there is a positive integer diagonal matrix $S$ such that $SB$ is skew-symmetric, where $S$ is said to be the {\bf skew-symmetrizer} of $B$. In this case, we say that $B$ is $S$-skew-symmetrizable.  For an $(m+n)\times n$ integer matrix $\tilde B=(b_{ij})$, the square submatrix $B=(b_{ij})_{1\leq i,j\leq n}$ is called the {\bf principal part} of $\tilde B$. Abusing terminology, we say that $\tilde B$ itself is
skew-symmetrizable or  skew-symmetric if its principal part $B$ is so.

\begin{Definition}
Let $\tilde B_{(m+n)\times n}=(b_{ij})$ be $S$-skew-symmetrizable, the mutation of  $\tilde B$  in the direction $k\in\{1,2,\cdots,n\}$ is the $(m+n)\times n$ matrix $\mu_k(\tilde B)=(b_{ij}^{\prime})$, where
\begin{eqnarray}\label{bmutation}
b_{ij}^{\prime}=\begin{cases}-b_{ij}~,& i=k\text{ or } j=k;\\ b_{ij}+sgn(b_{ik})max(b_{ik}b_{kj},0),&otherwise.\end{cases}
\end{eqnarray}
\end{Definition}

It is easy to see that $\mu_k(\tilde B)$ is still $S$-skew-symmetrizable, and $\mu_k(\mu_k(\tilde B))=\tilde B$.

\begin{Definition}\label{def1}
(i)  For $m,n>0$, an $m\times n$ integer matrix $A$ is called {\bf column  sign-coherent} (respectively, {\bf row  sign-coherent}) if any two nonzero entries of $A$ in the same column (respectively, row) have the same sign.

(ii)  Let $B_{1}$ be an $n\times n$ skew-symmetrizable matrix, and  $B_{2}\in M_{m\times n}(\mathbb Z)$ be a  column sign-coherent matrix. $B_{2}$ is called {\bf uniformly column sign-coherent with respect to $B_{1}$} if for any sequence of mutations $\mu_{k_s}\cdots\mu_{k_2}\mu_{k_1}$, the lower $m\times n$ submatrix  of  $\mu_{k_s}\cdots\mu_{k_2}\mu_{k_1}\begin{pmatrix}B_{1}\\B_{2}\end{pmatrix}$ is column sign-coherent.
\end{Definition}

\begin{Remark}
Note that the uniformly column sign-coherence of $B_{2}$ is invariant up to permutation of its row vectors, by the  equality  (\ref{bmutation}).
\end{Remark}

Given an $S$-skew-symmetrizable matrix $\tilde B=\begin{pmatrix} B\\I_n\end{pmatrix}\in M_{2n\times n}(\mathbb Z)$, let $\tilde B_\sigma=\begin{pmatrix} B_\sigma\\C_\sigma\end{pmatrix}$ be  the matrix obtained from $\tilde B$ by a sequence of mutations $\sigma:=\mu_{k_s}\cdots\mu_{k_2}\mu_{k_1}$. Recall that the lower part $C_\sigma$ of $\tilde B_\sigma$ is called a {\bf $C$-matrix} of $B$, see \cite{FZ3}.
\begin{Theorem}(\cite{GHKK})\label{fterm1} Using the above notations, for the skew-symmetrizable matrix $B\in M_n(\mathbb Z)$,  each  $C$-matrix $C_\sigma$ of $B$ is column sign-coherent.
\end{Theorem}
\begin{Remark}\label{rmk1}
 By Definition \ref{def1}, this theorem means  that $I_n$ is uniformly column sign-coherent with respect to the skew-symmetrizable matrix $B$.
\end{Remark}

Thanks to  Theorem \ref{fterm1}, one can define the sign functions on the column vectors of a {\em C-matrix} of a skew-symmetrizable matrix $B$.
 For a sequence of mutations $\sigma:=\mu_{k_s}\cdots\mu_{k_2}\mu_{k_1}$, denote by $\begin{pmatrix}B_{\sigma}\\ C_{\sigma}\end{pmatrix}:=\mu_{k_s}\cdots\mu_{k_2}\mu_{k_1}\begin{pmatrix}B\\I_n\end{pmatrix}$.
 If the entries of $j$-th column of $C_\sigma$ are all nonnegative (respectively, nonpositive), the sign of the $j$-th  column of $C_\sigma$ is defined as $\varepsilon_\sigma(j)=1$ (respectively, $\varepsilon_\sigma(j)=-1$).

\begin{Definition}
Let $C_\sigma$ be the C-matrix of $B$ given by a sequence of mutations $\sigma$, a column index $j\in \{1,\cdots,n\}$ of $C_\sigma$ is called {\bf green} (respectively, {\bf red}) if $\varepsilon_\sigma(j)=1$ (respectively, $\varepsilon_\sigma(j)=-1$).
\end{Definition}

Note that, by Theorem \ref{fterm1}, the column index of a {\em C-matrix} $C_\sigma$ is either green or red.

\begin{Definition} Let $B$ be a skew-symmetrizable matrix, and ${\bf k}=(k_1,\cdots,k_s)$ be a sequence of  column index of $B$. Denote by $C_{\sigma_j}$ the {C-matrix} of $B$ given by $\sigma_j:=\mu_{k_j}\cdots\mu_{k_2}\mu_{k_1}$.

(i)  ${\bf k}=(k_1,\cdots,k_s)$ is called a {\bf green-to-red sequence} of $B$  if each column index of the C-matrix $C_{\sigma_s}$ is red, i.e., $C_{\sigma_s}\in M_{n\times n}(\mathbb Z_{\leq 0})$.

(ii) ${\bf k}=(k_1,\cdots,k_s)$ is called a {\bf green sequence} of $B$ if  $k_i$ is green in the C-matrix $C_{\sigma_{i-1}}$ for $i=2,3,\cdots,s$.

(iii)  ${\bf k}=(k_1,\cdots,k_s)$ is called {\bf maximal green sequence}  of $B$  if it is both a green sequence  and a green-to-red sequence of $B$.

\end{Definition}

\begin{Example}\label{example1}
 Let $B=\begin{pmatrix}0&1&-1\\-1&0&1\\1&-1&0 \end{pmatrix}$, and ${\bf k}=(2,3,1,2)$.
\begin{eqnarray}
\begin{pmatrix} 0&1&-1\\-1&0&1\\1&-1&0\\  \hdashline[2pt/2pt] 1&0&0\\0&1&0\\0&0&1\end{pmatrix}\xrightarrow{\mu_2}\begin{pmatrix} 0&-1&0\\1&0&-1\\0&1&0\\ \hdashline[2pt/2pt] 1&0&0\\0&-1&1\\0&0&1\end{pmatrix}\xrightarrow{\mu_3}\begin{pmatrix}0&-1&0\\1&0&1\\0&-1&0\\  \hdashline[2pt/2pt] 1&0&0\\0&0&-1\\0&1&-1\end{pmatrix}
\xrightarrow{\mu_1}\begin{pmatrix}0&1&0\\-1&0&1\\0&-1&0\\  \hdashline[2pt/2pt] -1&0&0\\0&0&-1\\0&1&-1\end{pmatrix}
\xrightarrow{\mu_2}\begin{pmatrix}0&-1&1\\1&0&-1\\-1&1&0\\  \hdashline[2pt/2pt] -1&0&0\\0&0&-1\\0&-1&0\end{pmatrix}.\nonumber
\end{eqnarray}
Hence, ${\bf k}=(2,3,1,2)$ is a maximal green sequence of $B$.
\end{Example}

\section{Uniformly column sign-coherence of $B_2$}
In this section, we give a method to produce uniformly column sign-coherent matrices from a known one (Theorem \ref{thm1}). Then it is shown that all non-negative matrices and rank $\leq 1$ column sign-coherent matrices are uniform column sign-coherent (Corollary \ref{cor1} and Corollary \ref{cor2}).

\begin{Lemma}\label{uniformlem}
Let $P=(p_{ij})\in M_{p\times m}(\mathbb Z_{\geq0}),p,m>0$, and $B_1$ be an $n\times n$ skew-symmetrizable matrix. If $B_2\in M_{m\times n}(\mathbb Z)$ is column sign-coherent, then for $1\leq k\leq n$,
$$\mu_k(\begin{pmatrix}I_n&0\\0&P\end{pmatrix}\begin{pmatrix}B_1\\ B_2\end{pmatrix})=\mu_k\begin{pmatrix}B_1\\ PB_2\end{pmatrix}=\begin{pmatrix}I_n&0\\0&P\end{pmatrix}\mu_k\begin{pmatrix}B_1\\B_2 \end{pmatrix}.$$
\end{Lemma}

\begin{proof}
Denote by $\begin{pmatrix}B_1\\ B_2\end{pmatrix}=(b_{ij})$, $\mu_k\begin{pmatrix}B_1\\ B_2\end{pmatrix}=(b_{ij}^{\prime})$, $\begin{pmatrix}B_1\\ PB_2\end{pmatrix}=(a_{ij})$, $\mu_k\begin{pmatrix}B_1\\ PB_2\end{pmatrix}=(a_{ij}^{\prime})$.
Clearly, the principal parts of  $\mu_k\begin{pmatrix}B_1\\ PB_2\end{pmatrix}$ and $\begin{pmatrix}I_n&0\\0&P\end{pmatrix}\mu_k\begin{pmatrix}B_1\\B_2 \end{pmatrix}$ are equal. It suffices to show the lower parts of $\mu_k\begin{pmatrix}B_1\\ PB_2\end{pmatrix}$ and $\begin{pmatrix}I_n&0\\0&P\end{pmatrix}\mu_k\begin{pmatrix}B_1\\B_2 \end{pmatrix}$ are equal.
We know that for  $i>n, a_{ij}=\sum\limits_{l=1}^mp_{il}b_{n+l,j}$. By equation (\ref{bmutation}), for $i>n$,
$$a_{ij}^{\prime}=a_{ij}+sgn(a_{ik})max(a_{ik}b_{kj},0)
=\sum\limits_{l=1}^mp_{il}b_{n+l,j}+sgn(\sum\limits_{l=1}^mp_{il}b_{n+l,k})max(\sum\limits_{l=1}^mp_{il}b_{n+l,k}b_{kj},0).$$

Because $B_2$ is column sign-coherent and $P\in M_{p\times m}(\mathbb Z_{\geq0})$, we know that $(p_{il_1}b_{n+l_1,k})(p_{il_2}b_{n+l_2,k})\geq0, 1\leq l_1,l_2\leq m$. Thus if $p_{il_1}b_{n+l_1,k}\neq 0$, then $sgn(p_{il_1}b_{n+l_1,k})=sgn(\sum\limits_{l=1}^mp_{il}b_{n+l,k})$.
So
\begin{eqnarray}
a_{ij}^{\prime}&=&\sum\limits_{l=1}^mp_{il}b_{n+l,j}+sgn(\sum\limits_{l=1}^mp_{il}b_{n+l,k})max(\sum\limits_{l=1}^mp_{il}b_{n+l,k}b_{kj},0)\nonumber\\
&=&\sum\limits_{l=1}^mp_{il}b_{n+l,j}+\sum\limits_{l=1}^m sgn(p_{il}b_{n+l,k})max(p_{il}b_{n+l,k}b_{kj},0)\nonumber\\
&=&\sum\limits_{l=1}^mp_{il}(b_{n+l,j}+sgn(b_{n+l,k})max(b_{n+l,k}b_{kj},0)),\nonumber\\
&=&\sum\limits_{l=1}^mp_{il}b_{n+l,j}^{\prime}.\nonumber
\end{eqnarray}
Then the result follows.
\end{proof}

\begin{Theorem}\label{thm1}
Let $P\in M_{p\times m}(\mathbb Z_{\geq0})$ for $p,m>0$, and $B_1$ be an $n\times n$ skew-symmetrizable matrix. If $B_2\in M_{m\times n}(\mathbb Z)$ is uniformly  column sign-coherent with respect to $B_1$, then so is $PB_2$.
\end{Theorem}

\begin{proof}
For any sequence of mutation $\mu_{k_s}\cdots\mu_{k_2}\mu_{k_1}$,
the lower part of $\mu_{k_s}\cdots\mu_{k_2}\mu_{k_1}\begin{pmatrix}B_1\\B_2 \end{pmatrix}$ is column sign-coherent, by the uniformly  column sign-coherence of $B_2$ with respect to $B_1$.
Clearly, the lower part of $\begin{pmatrix}I_n&0\\0&P\end{pmatrix}\mu_{k_s}\cdots\mu_{k_2}\mu_{k_1}\begin{pmatrix}B_1\\B_2 \end{pmatrix}$ is also column sign-coherent.
By Lemma \ref{uniformlem}, we have
$$\mu_{k_s}\cdots\mu_{k_2}\mu_{k_1}(\begin{pmatrix}I_n&0\\0&P\end{pmatrix}\begin{pmatrix}B_1\\ B_2\end{pmatrix})=\begin{pmatrix}I_n&0\\0&P\end{pmatrix}\mu_{k_s}\cdots\mu_{k_2}\mu_{k_1}\begin{pmatrix}B_1\\B_2 \end{pmatrix}.$$
So the lower part of $\mu_{k_s}\cdots\mu_{k_2}\mu_{k_1}(\begin{pmatrix}I_n&0\\0&P\end{pmatrix}\begin{pmatrix}B_1\\ B_2\end{pmatrix})$ is also column sign-coherent.
Thus $PB_2$ is uniformly  column sign-coherent with respect to $B_1$.
\end{proof}

\begin{Corollary}\label{cor1}
 Let $B_1$ be an $n\times n$ skew-symmetrizable matrix. Then any matrix $P\in M_{m\times n}(\mathbb Z_{\geq0})$ is uniformly  column sign-coherent with respect to $B_1$.
\end{Corollary}
\begin{proof}
By Remark \ref{rmk1}, $I_n$ is uniformly  column sign-coherent with respect to $B_1$. Then the result follows from Theorem \ref{thm1} since $P=PI_n$.
\end{proof}

\begin{Corollary}\label{cor2}
Let $B_1$ be an $n\times n$ skew-symmetrizable matrix, and $B_2$ be an $m\times n$  column sign-coherent integer matrix. If $rank(B_2)\leq 1$, then $B_2$ is  uniformly column sign-coherent with respect to $B_1$.
\end{Corollary}
\begin{proof}
Because $rank(B_2)\leq 1$, $B_2$ has the form of
$$B_2=\begin{pmatrix}c_1\\ \vdots\\ c_n\end{pmatrix}\alpha,$$
 where $\alpha$ is a row vector, $c_1, c_2,\cdots,c_m\in\mathbb Q$.
  Because  $B_2$ is column sign-coherent, we can assume that $c_1,c_2,\cdots,c_m\geq 0$.
  Clearly, $\alpha$ is  uniformly column sign-coherent with respect to $B_1$.
Then by Theorem \ref{thm1}, $B_2$ is  uniformly column sign-coherent with respect to $B_1$.
\end{proof}

Following these two corollaries, there is a natural problem about  uniformly column sign-coherent matrices.
\begin{Problem}
Give all matrices $B_2$, which are  uniformly column sign-coherent with respect to $B_1$.
\end{Problem}

\begin{Proposition}\label{pro1}
Let $B=\begin{pmatrix}{B_1}_{n\times n}&{B_{3}}_{n\times m}\\{B_2}_{m\times n}&{B_{4}}_{m\times m} \end{pmatrix}$ be a skew-symmetrizable matrix. Then $B_2$ is  uniformly column sign-coherent with respect to $B_1$ if and only if  $B_{4}$ is invariant under
 any sequence of mutations $\mu_{k_s}\cdots\mu_{k_2}\mu_{k_1}$ with $1\leq k_i\leq n, i=1,2,\cdots,s$.
\end{Proposition}
\begin{proof}
Let $B=(b_{ij})$, and $\mu_k(B)=(b_{ij}^{\prime}), 1\leq k\leq n$. We know for any $i,j$, $$b_{ij}^{\prime}=b_{ij}+sgn(b_{ik})max(b_{ik}b_{kj},0). $$
Then $b_{ij}^{\prime}=b_{ij}$ if and only if $ b_{ik}b_{kj}\leq 0$, and then if and only if $ b_{ik}b_{jk}\geq 0$ because either  $b_{kj}b_{jk}<0$ or $b_{kj}=b_{jk}=0$ holds.

So, $B_{4}$ is invariant under the mutation $\mu_k(B)=(b_{ij}^{\prime}), 1\leq k\leq n$ if and only if $b_{ij}^{\prime}=b_{ij}$ for $n+1\leq i,j\leq n+m$, and then if and only if $ b_{ik}b_{jk}\geq 0$ for $n+1\leq i,j\leq n+m, 1\leq k\leq n$, which means that $B_2$ is  column sign-coherent. The result follows.
\end{proof}

\section{the existence of maximal green sequences}

Based on the discussion about uniformly column sign-coherence, in this section, we reduce the existence of maximal green sequences  for skew-symmetrizable matrices to the existence of maximal green sequences  for irreducible skew-symmetrizable matrices.

 \subsection{Irreducible skew-symmetrizable matrices}
 In this subsection, we give the definition of irreducible skew-symmetrizable matrices and their characteristic.

Let $B=(b_{ij})_{n\times n}$ be a  matrix, and $n_1, n_2$ be two positive integers. For $1\leq i_1<\cdots <i_{n_2}\leq n$ and $1\leq j_1<\cdots <j_{n_1}\leq n$, denote by $B_{j_1,\cdots,j_{n_1}}^{i_1,\cdots,i_{n_2}}$  the submatrix of $B$ with entries $b_{ij}$, where $i=i_1,\cdots,i_{n_2}$ and $j=j_1,\cdots,j_{n_1}$. If $n_2<n$ or $n_1<n$, the corresponding submatrix  $B_{j_1,\cdots,j_{n_1}}^{i_1,\cdots,i_{n_2}}$ is  a proper submatrix of $B$. If $n_2=n_1$ and $\{i_1,\cdots,i_{n_2}\}=\{j_1,\cdots,j_{n_1}\}$, the corresponding submatirx is a  principal submatirx of $B$.
Clearly, any principal submatrix of a skew-symmetrizable matrix is still skew-symmetrizable.

 \begin{Definition}\label{defind}
  A skew-symmetrizable matrix $B=(b_{ij})_{n\times n}$ is called {\bf reducible}, if $B$ has a proper submatrix $B_{j_1,\cdots,j_{n_1}}^{i_1,\cdots,i_{n_2}}$ satisfying£º

  (i). $B_{j_1,\cdots,j_{n_1}}^{i_1,\cdots,i_{n_2}}$ is a nonnegative matrix, i.e., $B_{j_1,\cdots,j_{n_1}}^{i_1,\cdots,i_{n_2}}\in M_{n_2\times n_1}(\mathbb Z_{\geq 0})$.

  (ii). $\{i_1,\cdots,i_{n_2}\}\cup \{j_1,\cdots,j_{n_1}\}=\{1,2,\cdots,n\}$ and $\{i_1,\cdots,i_{n_2}\}\cap \{j_1,\cdots,j_{n_1}\}=\phi$.

   Otherwise, $B$ is said to be {\bf irreducible} if such proper submatrix does not exist.

\end{Definition}
Clearly, $B$ is reducible if and only if  up to renumbering the row-column indexes of $B$, $B$ can be written as a block matrix as follows
$$B=\begin{pmatrix}{B_1}_{n_1\times n_1}&{B_{3}}_{n_1\times n_2}\\{B_2}_{n_2\times n_1}&{B_{4}}_{n_2\times n_2} \end{pmatrix},$$
such that the proper submatirx $B_2$ of $B$  is a nonnegative matrix, i.e.,  $B_2\in M_{n_2\times n_1}(\mathbb Z_{\geq 0})$.

In the skew-symmetric case, the definition of irreducibility for quiver version, has been given  in \cite{GM}.

For a skew-symmetrizable matrix $B$, we can encode the sign pattern of entries of $B$ by the quiver $\Gamma(B)$ with the vertices $1,2,\cdots,n$ and the arrows $i\rightarrow j$ for $b_{ij}>0$.  We call $\Gamma(B)$  the {\bf underlying quiver of $B$}. If  $\Gamma(B)$ is an acyclic quiver, then $B$ is said to be {\bf acyclic}. If $\Gamma(B)$ is a connected quiver, then  $B$ is said to be {\bf connected}. Clearly, if $B$ is an irreducible skew-symmetrizable matrix, then it must be connected.

For a quiver $Q$, if there exists a path from a vertex $a$ to a vertex $b$, then $a$ is said to be a {\bf predecessor} of $b$, and b is said to be a {\bf successor} of $a$.  For a vertex $a$ in $Q$, denote by $M(a)$, $N(a)$ the set of predecessors of $a$ and the set of successors of $a$ respectively.  Note that $a\in M(a)\cap N(a)$.

\begin{Proposition}\label{procycle}
Let $B=(b_{ij})_{n\times n}$ be a connected skew-symmetrizable matrix. Then $B$ is irreducible if and only if each  arrow of the  quiver $\Gamma(B)$ is in some oriented cycles.
\end{Proposition}
\begin{proof}
Suppose that $B$ is reducible, then $B$ can be written as a block matrix
$$B=\begin{pmatrix}{B_1}_{n_1\times n_1}&{B_{3}}_{n_1\times n_2}\\{B_2}_{n_2\times n_1}&{B_{4}}_{n_2\times n_2} \end{pmatrix},$$
such that $B_2\in M_{n_2\times n_1}(\mathbb Z_{\geq 0})$,  up to renumbering the row-column indexes of $B$. Since $B$ is connected, $B_2$ can not be a zero matrix. So there exist $i>n_1, j\leq n_1$ such that $b_{ij}\neq 0$. In fact $b_{ij}>0$, since  $B_2\in M_{n_2\times n_1}(\mathbb Z_{\geq 0})$.
We know that the arrow $i\rightarrow j$ is not in any oriented cycles of  $\Gamma(B)$, because $B_2\in M_{n_2\times n_1}(\mathbb Z_{\geq 0})$.

Suppose that there exists an arrow $i\rightarrow j$ is not in any oriented cycles of  $\Gamma(B)$.
We know that $i$ can not be a successor of $j$, i.e., $i\notin N(j)$. Let $n_1$ be the number of elements of $N(j)$. Clearly, $1\leq n_1\leq n-1$. We can renumber the row-column indexes of $B$ such that the elements of $N(j)$ are indexed by $1,2,\cdots,n_1$. $B$ can be written as a block matrix
$$B=\begin{pmatrix}{B_1}&{B_{3}}\\{B_2}&{B_{4}} \end{pmatrix}.$$

We claim that $B_2\in M_{(n-n_1)\times n_1}(\mathbb Z_{\geq 0})$.
Otherwise, there exists $k_1>n_1$ and $k_2\leq n_1$, i.e., $k_1\notin N(j), k_2\in N(j)$ such that $b_{k_1k_2}<0$. Thus $k_1$ is a successor of $k_2$, so is a successor of $j$, by $k_2\in N(j)$. This contradicts $k_1\notin N(j)$. So  $B_2\in M_{(n-n_1)\times n_1}(\mathbb Z_{\geq 0})$ and  $B$ is reducible. The proof is finished.
\end{proof}
\begin{Example}
Let $B=\begin{pmatrix}0&1&-1\\-2&0&2\\2&-2&0   \end{pmatrix}$. It is a skew-symmetrizable matrix. The quiver $\Gamma(B)$ is
$$
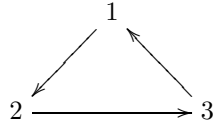
\begin{array}{cc}
\xymatrix{& 1\ar[ld]&
\\
 2\ar[rr]& &3\ar[lu]
}
\end{array}
$$
Since any arrow of $\Gamma(B)$ is in an oriented cycle, $B$ is irreducible.
\end{Example}

 \subsection{Reduction of the existence maximal green sequences }
 In this subsection, we
reduce the existence of maximal green sequences for skew-symmetrizable matrices to the existence
of a maximal green sequences for irreducible skew-symmetrizable matrices.

\begin{Lemma}\label{lem1}
Let $B$ be a skew-symmetrizable matrix and  $\sigma_{s+1}:=(k_1,\cdots,k_{s+1})$ be a sequence of column indexes of $B$. Denote by
 $\tilde B_{\sigma_{i}}=\begin{pmatrix}B_{\sigma_{i}}\\C_{\sigma_{i}}\end{pmatrix}:=\mu_{k_{i}}\cdots\mu_{k_2}\mu_{k_1}\begin{pmatrix}B\\I_n\end{pmatrix}$, $i=1,\cdots,s+1$. If $k_{s+1}$ is a green column index of $C_{\sigma_s}$, then any green column index $j$ of $C_{\sigma_s}$, with  $j\neq k_{s+1}$,  must be green in   $C_{\sigma_{s+1}}$.
\end{Lemma}
\begin{proof}
It can be proved in the same with that of Lemma 2.16 of \cite{BDP}. For the convenience of readers, we give  the  proof here.

 Because $j$ and $k_{s+1}$ are green column indexes of $C_{\sigma_s}$, we know that $(C_{\sigma_s})_{ij}\geq 0$ and $(C_{\sigma_s})_{ik_{s+1}}\geq 0$.
 By the definition of mutation, we have
 \begin{eqnarray}
 (C_{\sigma_{s+1}})_{ij}&=&(C_{\sigma_s})_{ij}+
 sgn((C_{\sigma_s})_{ik_{s+1}})max((C_{\sigma_s})_{ik_{s+1}}(C_{\sigma_s})_{k_{s+1}j},0)\nonumber\\
 &\geq&(C_{\sigma_s})_{ij}\geq 0.\nonumber
 \end{eqnarray}
So, $j$ is green in   $C_{\sigma_{s+1}}$.
\end{proof}

\begin{Theorem}\label{thm2}
Let $B=\begin{pmatrix}{B_{1}}_{n\times n}&{B_{3}}_{n\times m}\\{B_{2}}_{m\times n}&{B_{4}}_{m\times m} \end{pmatrix}=(b_{ij})$ be a skew-symmetrizable matrix, and $\tilde {\bf k}$ be a sequence $\tilde {\bf k}=(k_1,\cdots,k_s,k_{s+1},\cdots,k_{s+p})$, with $1\leq k_i\leq n$, and $n+1\leq k_j\leq m+n$ for $i=1,\cdots,s$, and $j=s+1,\cdots,s+p$, and denote by ${\bf k}=(k_1,\cdots,k_s)$, ${\bf j}=(k_{s+1},\cdots,k_{s+p})$.
If $B_{2}$ is in $M_{m\times n}(\mathbb Z_{\geq 0})$, then $\tilde {\bf k}$ is a maximal green sequence of $B$ if and only if ${\bf k}=(k_1,\cdots,k_s)$ (respectively, ${\bf j}=(k_{s+1},\cdots,k_{s+p})$) is a maximal green sequence of $B_{1}$ (respectively, $B_{4}$).

\end{Theorem}
\begin{proof}
 Let $\tilde B=\begin{pmatrix} {B_1}&{B_{3}}\\{B_2}&{B_{4}}\\ I_n&0\\0&I_m\end{pmatrix}$ and $B_{\sigma_i}=\mu_{k_i}\cdots \mu_{k_2}\mu_{k_1}(\tilde B), i=1,\cdots,s, s+1,\cdots, s+p$.
By $B_2\in M_{m\times n}(\mathbb Z_{\geq 0})$ and Corollary \ref{cor1}, we know that  $\begin{pmatrix}B_2\\I_n\\0\end{pmatrix}$ is uniformly column sign-coherent with respect to $B_1$.
By the same argument in Proposition \ref{pro1},
we know that the submatrix $\begin{pmatrix} {B_{4}}\\ 0\\I_m\end{pmatrix}$ of $\tilde B$ is invariant under the sequence of mutations $\mu_{k_s}\cdots\mu_{k_2}\mu_{k_1}, 1\leq k_i\leq n$ for $i=1,2,\cdots,s$. So
for $i\leq s$ the matrix $B_{\sigma_i}$ has the form of
\begin{eqnarray}\label{eqn1}
B_{\sigma_i}=\begin{pmatrix} {B_{1;\sigma_i}}&{B_{3;\sigma_i}}\\{B_{2;\sigma_i}}&{B_{4;\sigma_i}}\\ C_{1;\sigma_i}&0\\0&I_m\end{pmatrix}.
\end{eqnarray}

$``\Longleftarrow"$:
Because  ${\bf k}=(k_1,k_2,\cdots,k_s)$ is a maximal green sequence of $B_1$, we know that
 $C_{1;\sigma_s}\in M_{n\times n}(\mathbb Z_{\leq0})$. Thus  by the  uniformly column sign-coherence of $\begin{pmatrix}B_2\\I_n\\0\end{pmatrix}$  with respect to $B_1$, we know that $\begin{pmatrix} {B_{2;\sigma_s}}\\ C_{1;\sigma_s}\\0\end{pmatrix}\in M_{(2m+n)\times n}(\mathbb Z_{\leq0})$.
 By $B_{2;\sigma_s}\in M_{m\times n}(\mathbb Z_{\leq 0})$ and that the principal part of $B_{\sigma_s}$ is skew-symmetrizable, we can know $B_{3;\sigma_s}\in M_{n\times m}(\mathbb Z_{\geq0})$. Then by Corollary \ref{cor1}, we know that
$\begin{pmatrix} {B_{3;\sigma_s}}\\ 0\\I_m\end{pmatrix}\in M_{(2n+m)\times m}(\mathbb Z_{\geq0})$ is uniformly column sign-coherent with respect to $B_{4}$.
By the same argument in Proposition \ref{pro1} again, we know that the submatrix $\begin{pmatrix} {B_{1;\sigma_s}}\\ C_{1;\sigma_s}\\0\end{pmatrix}$ of $B_{\sigma_s}$ is invariant under the sequences of mutations $\mu_{k_{s+p}}\cdots\mu_{k_{s+2}}\mu_{k_{s+1}}(B_{\sigma_s}), n+1\leq k_i\leq n+m$ for $i=s+1,\cdots,s+p$.
So for $i\geq s+1$, the matrix $B_{\sigma_i}$ has the form of
$$B_{\sigma_i}=\begin{pmatrix} {B_{1;\sigma_s}}&{B_{3;\sigma_i}}\\{B_{2;\sigma_i}}&{B_{4;\sigma_i}}\\ C_{1;\sigma_s}&0\\0&C_{4;\sigma_i}\end{pmatrix}.$$
Because ${\bf j}=(j_1,j_2,\cdots,j_p)$ is a maximal green sequence of $B_{4}$, we know that $C_{4;\sigma_{s+p}}\in M_{m\times m}(\mathbb Z_{\leq 0})$.
Thus the lower part of $B_{\sigma_{s+p}}$ is $\begin{pmatrix}  C_{1;\sigma_s}&0\\0&C_{4;\sigma_{s+p}}\end{pmatrix}\in M_{(m+n)\times(m+n)}(\mathbb Z_{\leq 0})$.
It can be seen that  $\tilde {\bf k}=({\bf k},{\bf j})$ is a green sequence of $B$, so it is maximal.

$``\Longrightarrow"$£ºBy  (\ref{eqn1}),
$B_{\sigma_s}=\begin{pmatrix} {B_{1;\sigma_s}}&{B_{3;\sigma_s}}\\{B_{2;\sigma_s}}&{B_{4;\sigma_i}}\\ C_{1;\sigma_s}&0\\0&I_m\end{pmatrix}.$
Clearly, ${\bf k}=(k_1,\cdots,k_s)$  is a green sequence of $B_1$ and  ${\bf j}=(k_{s+1},\cdots,k_{s+p})$ is a maximal green sequence of $B_{4}$.

We claim that each $l\in\{1,2,\cdots,n\}$ is red in $C_{1;\sigma_s}$, i.e., $C_{1;\sigma_s}\in M_{n\times n}(\mathbb Z_{\leq0})$,  and thus ${\bf k}=(k_1,\cdots,k_s)$  is a maximal green sequence of $B_1$.
Otherwise, there will exist a $l_0\in\{1,2,\cdots,n\}$ which is green in  $C_{1;\sigma_s}$. Thus $l_0$ is green in $\begin{pmatrix} C_{1;\sigma_s}&0\\0&I_m\end{pmatrix}$ the lower part of $B_{\sigma_s}$.
By Lemma \ref{lem1} and $l_0\leq n<k_{s+i}, i=1,2,\cdots,p$, we know that $l_0$ will remain green  in  $\begin{pmatrix} C_{1;\sigma_{s+p}}&C_{3;\sigma_{s+p}}\\C_{2;\sigma_{s+p}}&C_{4;\sigma_{s+p}}\end{pmatrix}$ the lower part of $B_{\sigma_{s+p}}$.
It is impossible since $(k_1,\cdots,k_{s},k_{s+1},\cdots,k_{s+p})$ is a maximal green sequence of $B$.
\end{proof}

When $B$ is skew-symmetric and $B_2$ is a matrix over $\{0,1\}$, the above theorem has been actually given in Theorem 3.12 of \cite{GM}. The authors of \cite{GM} believed that the result also holds for $B_2\in M_{m\times n}(\mathbb Z_{\geq 0})$, but they did not have a proof. We in fact have given the proof for this in the skew-symmetrizable case.

\begin{Remark}\label{rmkqin}
Note that the $"\Longleftarrow"$ part of the proof of the above theorem also holds if we replace maximal green sequences with green-to-red sequences, and the proof is identical.  We are thankful to Fan Qin for pointing out this.
\end{Remark}
\begin{Example}\label{example2}
Let $B=\begin{pmatrix}0&-2\\3&0\end{pmatrix}$. Here $B_1=0=B_{4}$, $B_2=3\geq0$.   The column index set of $B_1$ is $\{1\}$ and  the column index set of $B_{4}$ is $\{2\}$. It is known that $(1)$ is a maximal green sequence of $B_1$ and $(2)$ is a maximal green sequence of $B_{4}$. Then by Theorem \ref{thm2}, $(1,2)$ is a maximal green sequence of $B$. Indeed,
$$\begin{pmatrix}0&-2\\3&0\\ \hdashline[2pt/2pt] 1&0\\0&1\end{pmatrix}\xrightarrow{\mu_1}\begin{pmatrix}0&2\\-3&0\\  \hdashline[2pt/2pt] -1&0\\0&1\end{pmatrix} \xrightarrow{\mu_2}\begin{pmatrix}0&-2\\3&0\\  \hdashline[2pt/2pt] -1&0\\0&-1\end{pmatrix}.$$

\end{Example}

\begin{Example}
 Let $B=\begin{array}{c@{\hspace{-5pt}}l}
 \left(
\begin{array}{ccc;{2pt/2pt}cccc;{2pt/2pt}cccc;{2pt/2pt}c}
 0&1&-1&-2&-2\\
 -1&0&1&0&-4\\
 1&-1&0&-3&0\\
   \hdashline[2pt/2pt]
2&0&3&0&-2  \\
1&2&0&1&0
   \end{array}
\right)=\begin{pmatrix}B_1&B_{3}\\B_2&B_{4}\end{pmatrix}
\end{array}$\;  where $B_1$ is of order $3\times 3$ and $B_4$ is of order $2\times 2$. Clearly, $B$ is skew-symmetrizable with skew-symmetrizer $S=diag\{1,1,1,1,2\}$ and $B_2\in M_{2\times 3}(\mathbb Z_{\geq 0})$. The column index set of $B_1$ is  $\{1,2,3\}$ and the column index set of $B_{4}$ is $\{4,5\}$. By Example \ref{example1} (respectively, Example \ref{example2}), $(2,3,1,2)$ (respectively, $(4,5)$) is a maximal green sequence of $B_1$ (respectively, $B_{4}$). Then by Theorem \ref{thm2}, $(2,3,1,2,4,5)$ is a maximal green sequence of $B$. Indeed,

\begin{small}
\begin{eqnarray}
\tilde B:=&&
\begin{array}{c@{\hspace{-5pt}}l}
 \left(
\begin{array}{ccc;{2pt/2pt}cccc;{2pt/2pt}cccc;{2pt/2pt}c}
 0&1&-1&-2&-2\\
 -1&0&1&0&-4\\
 1&-1&0&-3&0\\
   \hdashline[2pt/2pt]
2&0&3&0&-2  \\
1&2&0&1&0\\
 \hdashline[2pt/2pt]
 1&0&0&0&0\\
 0&1&0&0&0\\
 0&0&1&0&0\\
 \hdashline[2pt/2pt]
 0&0&0&1&0\\
 0&0&0&0&1
   \end{array}
\right)
\end{array}
 \xrightarrow{\mu_2}
\begin{array}{c@{\hspace{-5pt}}l}
 \left(
\begin{array}{ccc;{2pt/2pt}cccc;{2pt/2pt}cccc;{2pt/2pt}c}
0&-1&0&-2&-2\\
1&0&-1&0&4\\
0&1&0&-3&-4\\
 \hdashline[2pt/2pt]
2&0&3&0&-2\\
1&-2&2&1&0\\
 \hdashline[2pt/2pt]
1&0&0&0&0\\
0&-1&1&0&0\\
0&0&1&0&0\\
 \hdashline[2pt/2pt]
0&0&0&1&0\\
0&0&0&0&1
   \end{array}
\right)
\end{array}
\xrightarrow{\mu_3}
\begin{array}{c@{\hspace{-5pt}}l}
 \left(
\begin{array}{ccc;{2pt/2pt}cccc;{2pt/2pt}cccc;{2pt/2pt}c}
0&-1&0&-2&-2\\
1&0&1&-3&0\\
0&-1&0&3&4\\
\hdashline[2pt/2pt]
2&3&-3&0&-2\\
1&0&-2&1&0\\
\hdashline[2pt/2pt]
1&0&0&0&0\\
0&0&-1&0&0\\
0&1&-1&0&0\\
\hdashline[2pt/2pt]
0&0&0&1&0\\
0&0&0&0&1
   \end{array}
\right)
\end{array}
\xrightarrow{\mu_1}
\nonumber\\
&&
\begin{array}{c@{\hspace{-5pt}}l}
 \left(
\begin{array}{ccc;{2pt/2pt}cccc;{2pt/2pt}cccc;{2pt/2pt}c}
0&1&0&2&2\\
-1&0&1&-3&0\\
0&-1&0&3&4\\
\hdashline[2pt/2pt]
-2&3&-3&0&-2\\
-1&0&-2&1&0\\
\hdashline[2pt/2pt]
-1&0&0&0&0\\
0&0&-1&0&0\\
0&1&-1&0&0\\
\hdashline[2pt/2pt]
0&0&0&1&0\\
0&0&0&0&1
   \end{array}
\right)
\end{array}
\xrightarrow{\mu_2}
\begin{array}{c@{\hspace{-5pt}}l}
 \left(
\begin{array}{ccc;{2pt/2pt}cccc;{2pt/2pt}cccc;{2pt/2pt}c}
0&-1&1&2&2\\
1&0&-1&3&0\\
-1&1&0&0&4\\
\hdashline[2pt/2pt]
-2&-3&0&0&-2\\
-1&0&-2&1&0\\
\hdashline[2pt/2pt]
-1&0&0&0&0\\
0&0&-1&0&0\\
0&-1&0&0&0\\
\hdashline[2pt/2pt]
0&0&0&1&0\\
0&0&0&0&1
   \end{array}
\right)
\end{array}
\xrightarrow{\mu_4}
\begin{array}{c@{\hspace{-5pt}}l}
 \left(
\begin{array}{ccc;{2pt/2pt}cccc;{2pt/2pt}cccc;{2pt/2pt}c}
0&-1&1&-2&2\\
1&0&-1&-3&0\\
-1&1&0&0&4\\
\hdashline[2pt/2pt]
2&3&0&0&2\\
-1&0&-2&-1&0\\
\hdashline[2pt/2pt]
-1&0&0&0&0\\
0&0&-1&0&0\\
0&-1&0&0&0\\
\hdashline[2pt/2pt]
0&0&0&-1&0\\
0&0&0&0&1
   \end{array}
\right)
\end{array}
\xrightarrow{\mu_5}\nonumber\\
&&
\begin{array}{c@{\hspace{-5pt}}l}
 \left(
\begin{array}{ccc;{2pt/2pt}cccc;{2pt/2pt}cccc;{2pt/2pt}c}
0&-1&1&-2&-2\\
1&0&-1&-3&0\\
-1&1&0&0&-4\\
\hdashline[2pt/2pt]
2&3&0&0&-2\\
1&0&2&1&0\\
\hdashline[2pt/2pt]
-1&0&0&0&0\\
0&0&-1&0&0\\
0&-1&0&0&0\\
\hdashline[2pt/2pt]
0&0&0&-1&0\\
0&0&0&0&-1
   \end{array}
\right).
\end{array}\nonumber
\end{eqnarray}
\end{small}
\end{Example}

Denote by $\tilde B^{\prime}=\mu_2\mu_2\mu_3\mu_2(\tilde B)$. It is can be seen that the submatrix $\tilde B_{4,5}^{4,5,6,7,8,9,10}$ of $\tilde B$ is invariant along the mutation sequence $(2,3,1,2)$ and the submatrix $\tilde {B^{\prime}}_{1,2,3}^{1,2,3,6,7,8,9,10}$ of $\tilde B^{\prime}$ is invariant along the mutation sequence $(4,5)$.

The following lemma is the skew-symmetrizable version of Theorem 9  and Theorem 17 of \cite{M}.  Although these results in \cite{M} were verified for the situation of quivers, or say, in skew-symmetric case,  the method of their proofs in \cite{M} can be naturally extended to the skew-symmetrizable case.

\begin{Lemma}\label{subthm}
Let $B$ be a skew-symmetrizable matrix. If $B$ admits a maximal green sequence (respectively, green-to-red sequence), then any principal submatrix of $B$ also has a maximal green sequence (respectively, green-to-red sequence).
\end{Lemma}

\begin{Theorem}\label{cormgs}
Let $B$ be a skew-symmetrizable matrix.  Then $B$ has a maximal green sequence (respectively, green-to-red sequence) if and only if any irreducible principal submatrix of $B$ has a maximal green sequence (respectively, green- to-red sequence).
\end{Theorem}
\begin{proof}
It follows from Lemma \ref{subthm}, Theorem \ref{thm2} and Remark \ref{rmkqin}.
\end{proof}

\begin{Remark}
By the above theorem, we can give our explanation of the existence of maximal green sequences for acyclic skew-symmetrizable matrices. Because any irreducible principal submatrix of an acyclic skew-symmetrizable matrix $B$ is only a $1\times 1$ zero matrix, and  it always has a maximal green sequence, we then know that by Theorem \ref{cormgs} any acyclic skew-symmetrizable matrix admits a maximal green sequence.
\end{Remark}

 By Theorem \ref{cormgs},  we  reduce the existence of maximal green sequences (respectively, green-to-red sequences) for skew-symmetrizable matrices to the existence of maximal green
sequences (respectively, green-to-red sequences) for irreducible skew-symmetrizable matrices $B$, i.e. those $B$
 whose all arrows of $\Gamma(B)$ are in oriented cycles, by  Proposition \ref{procycle}.
 In \cite{GM}, the authors classified the irreducible principal submatrices of the skew-symmetric matrices of type $A$
  and proved any such an irreducible matrix has a maximal green sequence. Therefore, the authors get that any  skew-symmetric matrix of type $A$ has a maximal green sequence. Inspired by this and Theorem \ref{thm2}, we propose the following problem as an attempt to end the discussion on the existence of maximal green sequences (respectively, green-to-red sequences).
\begin{Problem}
 When does an irreducible skew-symmetrizable matrix  admit a maximal green sequence (respectively, green-to-red sequence)?
\end{Problem}

{\bf Acknowledgements:}\; This project is supported by the National Natural Science Foundation of China (No.11671350 and No.11571173).


\end{document}